\documentclass[11pt,a4paper]{amsart}
\usepackage{amsmath}
\usepackage[utf8]{inputenc}
\usepackage[english]{babel}
\usepackage[T1]{fontenc}
\usepackage{url}
\usepackage{mathrsfs}
\usepackage{ifthen}

\usepackage{tikz}
\usepackage{xcolor}
\usetikzlibrary{knots}

\usepackage{ amssymb }
\usepackage{lscape}
\usepackage{multirow}
\usepackage{wrapfig} 

\usepackage{caption}

\newtheorem{thm}{Theorem}
\newtheorem{corollary}{Corollary}

\theoremstyle{definition}
\newtheorem{definition}{Definition}

\theoremstyle{remark}
\newtheorem{remark}{Remark}

\newcommand{\crv}{\operatorname{cr}^v}
\newcommand{\R}{\mathbb{R}}

\title{Meander diagrams of virtual knots}
\author{Y. Belousov, V. Chernov, A. Malyutin, R. Sadykov}
\thanks{This work was supported by the Russian Science Foundation under grant no. 22-11-00299.}
\begin{document}
\begin{abstract}
For classical knots, there is a concept of (semi)meander diagrams; in this short note we generalize this concept to virtual knots and prove that the classes of meander and semimeander diagrams are universal (this was known for classical knots). We also introduce a new class of invariants for virtual knots --- virtual $k$-arc crossing numbers and we use Manturov projection to show that for all classical knots the virtual $k$-arc crossing number equals to the classical $k$-arc crossing number. 
\end{abstract}
\maketitle

\section*{Introduction}
A class of knot diagrams is called \emph{universal} if every knot has a diagram within this class. For classical knots, several universal classes of diagrams are known. Arguably, the most famous universal classes are the diagrams obtained as closures of braids (as stated in the classical Alexander's theorem~\cite{A1923}), and the so-called rectangular diagrams (\cite{C95}). Other examples of universal classes include clasp-diagrams (\cite{MM19}), potholders (\cite{EZHLN19}), and stars (\cite{A15}). More examples of universal classes of knot diagrams can be found in~\cite{EZHLN19}. In this paper we focus on another universal class of diagrams ---  so-called semimeander diagrams. 

A diagram of a classical knot is called \emph{semimeander} if it is composed of two smooth simple arcs. Apparently, the fact that each classical knot admits a semimeander diagram was first proved by G.~Hots in 1960 (\cite{H60}) and has been rediscovered in various forms by several authors since then. Moreover, this result admits a refinement: each classical knot also has a \emph{meander diagram}, i.e., a diagram composed of two smooth simple arcs whose common endpoints lie on the boundary of the diagram's convex hull. We refer the reader to~\cite{BM20} for a detailed historical background on semimeander and meander diagrams of classical knots. 

There are several applications of meander and semimeander diagrams; for example, they have been used to demonstrate the universality of the class of potholder diagrams (\cite{EZHLN19}), to introduce a new family of knot invariants (\cite{B20}), and to establish some new results about band surgeries (\cite{BKMMF22}).

In this paper, we extend the notion of semimeander (and meander) diagrams to the case of virtual knots. 
Virtual knots were introduced by L.~Kauffman~\cite{Kauf99} as a generalization of classical knots by allowing crossings that can be either real or virtual. Virtual knots can be viewed as knots in thickened surfaces up to isotopy and stabilization and destabilization of surfaces away from the knot projection. They can also be viewed as the equivalence classes of Gauss diagrams up to formal chord diagrams versions of the Reidemeister moves. 

Similar to classical knots, virtual knots can be represented via diagrams with a finite number of crossings, each crossing being either a classical crossing or a virtual crossing. Virtual crossings are depicted as small circles around crossing points. Two virtual knot diagrams are equivalent if they are related by a series of moves: the three classical Reidemeister moves (which apply to classical crossings) and four additional virtual Reidemeister moves (which handle virtual crossings). We refer the reader to~\cite{Kauf99} for details. 

The definition of universality is applied literally to virtual knots: a class of virtual knot diagrams is called \emph{universal} if every virtual knot has a diagram within this class. In this paper, we prove that the classes of semimeander and meander diagrams of virtual knots are universal; we also raise some questions related to these concepts. 

\section{Semimeander diagrams of virtual knots}
The notion of semimeander diagrams for virtual knots can be generalized in two ways.
\begin{definition}
    A diagram of a virtual knot is called \emph{strong semimeander} if it is composed of two smooth simple arcs.
\end{definition}

\begin{definition}\label{def: semimeander diagram}
    A diagram of a virtual knot is called \emph{semimeander} if it is composed of two smooth arcs, all self-intersections of which are virtual crossings.
\end{definition}

\begin{thm}\label{thm: strong semimeander diag exists}
    Each virtual knot has a strong semimeander diagram.
\end{thm}
\begin{proof}
    Let $D$ be a diagram of a virtual knot. Let us choose some simple arc $I$ in $D$ such that no endpoint of $I$ is a crossing. If there are any crossings (both classical or virtual) that do not lie on $I$, we make some local transformations of $D$ so that the knot does not change but the number of crossings not lying on $I$ decreases. Repeating this procedure yields a strong semimeander diagram. The involved local moves are presented in Fig.~\ref{fig: x-moves} (they are sequences of Reidemeister moves, so they do not change the equivalence class of the virtual knot). 

    \begin{figure}
    \centering
    \begin{tikzpicture}
\node (v_s) at (1,0){};
\node (v_e) at (4, 0){};

\node (x1) at (3,-0.7){};
\node (x2) at (3, 0.7){};

\node (b1) at (2,0.7){};
\node (b2) at (2,-0.7){};

\draw [line width = 3, white, double=black, double distance = 1.5] (x1.center) to (x2.center);
\draw [line width = 3, white, double=black, double distance = 1.5] (v_s.center) to (v_e.center);
\draw [line width = 3, white, double=black, double distance = 1.5] (b1.center) to (b2.center);
\end{tikzpicture}
\hspace{0.4cm}
\raisebox{0.8cm}{
\scalebox{0.88}{
\begin{tikzpicture}
%\draw [white, ultra thick] (,0) circle [radius=0.1];;
\node(r) at (-0.15, -1.7) {}; 
\draw[->, thin] (-1.5,-2)--(1.2,-2);
\end{tikzpicture}}}
\hspace{0.6cm}
\begin{tikzpicture}
\node (v_s) at (1,0){};
\node (v_e) at (4, 0){};

\node (x1) at (3,-0.7){};
\node (xm1) at (1.7,-0.3){};
\node (xm2) at (1.7,0.3){};
\node (x2) at (3, 0.7){};

\node (b1) at (2,0.7){};
\node (b2) at (2,-0.7){};

\draw [line width = 3, white, double=black, double distance = 1.5] (x1.center) 
	to [out = 90, in = 0, looseness = 0.8] (xm1.center)
	to [out = 180, in = 180, looseness = 2] (xm2.center)
	to [out = 0, in = -90, looseness = 0.8] (x2.center);

\draw [line width = 3, white, double=black, double distance = 1.5] (v_s.center) to (v_e.center);
\draw [line width = 3, white, double=black, double distance = 1.5] (b1.center) to (b2.center);
\end{tikzpicture}

\begin{tikzpicture}
\node (v_s) at (1,0){};
\node (v_e) at (4, 0){};

\node (b1) at (3,-0.7){};
\node (b2) at (3, 0.7){};

\node (x1) at (2,0.7){};
\node (x2) at (2,-0.7){};

\draw [line width = 3, white, double=black, double distance = 1.5] (v_s.center) to (v_e.center);
\draw [line width = 3, white, double=black, double distance = 1.5] (b1.center) to (b2.center);
\draw [line width = 1.5] (x1.center) to (x2.center);
\draw[fill = white] (2, 0) circle [radius = 0.07]; 
\end{tikzpicture}
\hspace{0.4cm}
\raisebox{0.8cm}{
\scalebox{0.88}{
\begin{tikzpicture}
%\draw [white, ultra thick] (,0) circle [radius=0.1];;
\node(r) at (-0.15, -1.7) {}; 
\draw[->, thin] (-1.5,-2)--(1.2,-2);
\end{tikzpicture}}}
\hspace{0.6cm}
\begin{tikzpicture}
\node (v_s) at (1,0){};
\node (v_e) at (4, 0){};

\node (x1) at (3,-0.7){};
\node (xm1) at (1.7,-0.3){};
\node (xm2) at (1.7,0.3){};
\node (x2) at (3, 0.7){};

\node (b1) at (2,0.7){};
\node (b2) at (2,-0.7){};

\draw [line width =  1.5] (v_s.center) to (v_e.center);

\draw [line width = 3, white, double=black, double distance = 1.5] (x1.center) 
	to [out = 90, in = 0, looseness = 0.8] (xm1.center)
	to [out = 180, in = 180, looseness = 2] (xm2.center)
	to [out = 0, in = -90, looseness = 0.8] (x2.center);
\draw [line width = 1.5] (b1.center) to (b2.center);
\draw[fill= white] (2, 0.3) circle [radius = 0.07]; 
\draw[fill= white] (2, -0.3) circle [radius = 0.07]; 
\draw[fill= white] (2, 0) circle [radius = 0.07]; 
\end{tikzpicture}

\begin{tikzpicture}
\node (v_s) at (1,0){};
\node (v_e) at (4, 0){};

\node (x1) at (3,-0.7){};
\node (x2) at (3, 0.7){};

\node (b1) at (2,0.7){};
\node (b2) at (2,-0.7){};

\draw [line width = 1.5] (v_s.center) to (v_e.center);
\draw [line width = 1.5] (b1.center) to (b2.center);
\draw [line width = 1.5] (x1.center) to (x2.center);
\draw[fill= white] (2, 0) circle [radius = 0.07]; 
\draw[fill= white] (3, 0) circle [radius = 0.07]; 
\end{tikzpicture}
\hspace{0.4cm}
\raisebox{0.8cm}{
\scalebox{0.88}{
\begin{tikzpicture}
%\draw [white, ultra thick] (,0) circle [radius=0.1];;
\node(r) at (-0.15, -1.7) {}; 
\draw[->, thin] (-1.5,-2)--(1.2,-2);
\end{tikzpicture}}}
\hspace{0.6cm}
\begin{tikzpicture}
\node (v_s) at (1,0){};
\node (v_e) at (4, 0){};

\node (x1) at (3,-0.7){};
\node (xm1) at (1.7,-0.3){};
\node (xm2) at (1.7,0.3){};
\node (x2) at (3, 0.7){};

\node (b1) at (2,0.7){};
\node (b2) at (2,-0.7){};

\draw [line width = 1.5] (v_s.center) to (v_e.center);
\draw [line width = 1.5] (b1.center) to (b2.center);

\draw [line width = 1.5] (x1.center) 
	to [out = 90, in = 0, looseness = 0.8] (xm1.center)
	to [out = 180, in = 180, looseness = 2] (xm2.center)
	to [out = 0, in = -90, looseness = 0.8] (x2.center);
\draw[fill= white] (2, 0.3) circle [radius = 0.07]; 
\draw[fill= white] (2, -0.3) circle [radius = 0.07]; 
\draw[fill= white] (1.355, 0) circle [radius = 0.07]; 
\draw[fill= white] (2, 0) circle [radius = 0.07]; 
\end{tikzpicture}
    \caption{Examples of $\mathfrak{X}$-moves}
    \label{fig: x-moves}
\end{figure}
    
    Now let us describe this process in more detail. Let us choose some parametrization $f:S^1\to D$ of $D$, and let $\delta \subset S^1$ be an arc such that $f\left(\delta\right) = I$. Let $X = \{t \in S^1\,:\,f(t) \text{ is a crossing in }D\setminus I\}$. We can find $t_*\in X$ with the following property: (\textasteriskcentered) there exists an arc $c\subset S^1$ connecting $t_*$ and $\delta$ such that $c \cap X = \{t_*\}$. In this case, $f(c)$ is a simple arc in $D$ that contains only crossings lying on $I$. So we can perform a series of transformations from Fig.~\ref{fig: x-moves}. As a result, the number of crossings not lying on $I$ will decrease. 
\end{proof}

\begin{corollary}
    Each virtual knot has a semimeander diagram. 
\end{corollary}
\begin{thm}    
    If a virtual knot $K$ has a diagram with $n$ crossings, then there exists a strong semimeander diagram of $K$ with at most $\sqrt{3}^n$ crossings. 
\end{thm}
\begin{proof}
    Let $D$ be a diagram of $K$ with $n$ crossings. We can assume that $D$ is reduced (in the sense that no crossing can be eliminated through a single Reidemeister move). Then there exists a simple arc $I$ in $D$ containing $k \geq \min\{6, n\}$ crossings of $D$ (see~\cite[Theorem~2]{BM17}). If $n \leq 6$, the diagram is a strong semimeander itself, thus without loss of generality we assume that $n > 6$.
    
    We will use the notation from the proof of Theorem~\ref{thm: strong semimeander diag exists}. If $n - k = 1$, then $X$ consists of a single element and $D$ contains two distinct simple arcs connecting this element with $\delta$. In this case, we can choose an arc $c$ of these two for which the number $|f(c) \cap I|$ is the smallest (this number is at most $\frac{k}{2}$). If $n - k > 1$, then there are two distinct elements $t_1$ and $t_2$ in $X$ both having property (\textasteriskcentered). We need to consider two cases, whether $f(t_1) = f(t_2)$ or not. 
    \begin{itemize}
        \item If $f(t_1)\neq f(t_2)$, then we can simultaneously perform the discussed transformation for $t_1$ and for $t_2$. As a result, the number of crossings not lying on $I$ will decrease by two, while the number of crossings in the diagram will increase by at most $2k$. 
        \item Assume $f(t_1) = f(t_2)$. In this case, there exist two closed arcs $\gamma_1,\ \gamma_2 \subset S^1$ such that for each $i\in \{1,2\}$ 
        \begin{enumerate}
            \item $t_i \in \gamma_i$, 
            \item $|\gamma_i \cap X|=2$, 
            \item one of the endpoints of $\gamma_i$ lies in $X$, and 
            \item $\gamma_i$ and $\delta$ share a single endpoint. 
        \end{enumerate}
        Let $s_i$ be the only element in  $(\gamma_i \cap X)\setminus \{t_i\}$. Since $f(s_i) \neq f(t_i)$, the arc $f(\gamma_i)$ is simple, and we can sequentially perform a series of $\mathfrak{X}$-moves either for $f(t_1)$ and $f(s_1)$ or for $f(t_2)$ and $f(s_2)$. In order to add less crossings, we make a choice depending on which of $|f(\gamma_1)\cap I|+ |f(c_1)\cap I|$  or  $|f(\gamma_2)\cap I|+ |f(c_2)\cap I|$ is smaller, where $c_i$ is a subarc of $\gamma_i$ connecting $\delta$ and $t_i$. Thus we can decrease the number of crossings not lying on $I$ by two, while the number of crossings in the diagram will increase by at most $2k$. 
    \end{itemize}
  
    As a result, we obtain a strong semimeander diagram with at most
    $$
        3^{\lfloor\frac{n-k}{2} \rfloor}\left(1 + \frac{1}{2}\right)^{n\ \mathrm{mod}\ 2}k \leq  3^{\frac{n-k}{2}}k \leq 3^{\frac{n-6}{2}}\cdot6 < \sqrt{3}^n
    $$
    crossings.
\end{proof}

\begin{remark}    
There are more efficient algorithms for constructing semimeander diagrams of classical knots (see~\cite{B23}). These algorithms use transformations similar to flypes and therefore they likely cannot be directly applied to the case of virtual knots (as virtual flypes can change the knot type, see, for example,~\cite{ZJZ04} and~\cite{K17}). It is an interesting problem to find efficient ways to construct semimeander diagrams for virtual knots.
\end{remark}

Since it is convenient to work with virtual knots using Gauss diagrams, we think that Definition~\ref{def: semimeander diagram} is a better generalization of the classical notion of a semimeander diagram to the case of virtual knots. Indeed, a diagram $D$ of a virtual knot is semimeander if and only if the Gauss diagram constructed from $D$ can be composed of two circle arcs in such a way that each chord starts on one arc and ends on the other. 

\subsection*{Connection with bridge index}
Reformulation of the definition of a semimeander diagram in terms of Gauss diagrams indicates a connection with the virtual bridge index of a virtual knot. For example, there are non-trivial virtual knots with bridge index one (see~\cite{BC08}), and each one-bridge diagram is a semimeander diagram. For the case of two-bridge knots, it was a conjecture due to Jablan and Radovi\'{c} (see~\cite{JR15}) that each two-bridge classical knot has a minimal diagram that is semimeander. This conjecture was proved in~\cite{BM20}. It would be interesting to get a similar result for virtual knots.

\section{Meander diagrams of virtual knots}
As with semimeander diagrams, the concept of a meander diagram can be generalized in two ways.
\begin{definition}
    A diagram of a virtual knot is called \emph{strong meander} if it is composed of two smooth simple arcs whose common endpoints lie on the boundary of the convex hull of the diagram. 
\end{definition}

\begin{definition}
    A diagram of a virtual knot is called \emph{meander} if it is composed of two smooth arcs whose common endpoints lie on the boundary of the convex hull of the diagram and all self-intersections of these arcs are virtual crossings. 
\end{definition}

\begin{thm}
    Each virtual knot has a strong meander diagram.
\end{thm}
\begin{proof}
    Let $K$ be a virtual knot, and let $D$ be its strong semimeander diagram composed of smooth simple arcs $C_1$ and $C_2$. We denote the common endpoints of $C_1$ and $C_2$ by $x_1$ and $x_2$. We can assume that $C_1$ is a line segment and that $x_1$ lies on the boundary of the convex hull of $D$. Now let us introduce some additional notation. 
    \begin{itemize}
        \item Let $L$ be the line containing $C_1$.
        \item Let $R$ be the ray starting at $x_1$ and containing $C_1$.
        \item Let $A$ be a small subarc of $C_2$ with an endpoint $x_2$ such that $A$ contains no crossings of $D$.
        \item Let $a$ be the endpoint of $A$ distinct from $x_2$.
        \item Let $x_3$ be any point on $R$ such that the segment $[x_1;\,x_3]$ contains $R\cap D$.
        \item Let $R' = L \setminus (x_1;\, x_3]$. 
    \end{itemize}
    See Fig.~\ref{fig: constructing meander diagram}, where some of the defined objects are presented.
    
    From a general position argument, we can assume that $R \cap C_2$ is a finite set.  
    There exists a simple arc $A'$ in $(\R^2 \setminus (C_2 \cup R')) \cup \{a\}$ with endpoints $a$ and $x_3$ such that $A'\cap [x_1;\,x_3]$ is a finite set. 
    Such an arc can always be found because $(\R^2 \setminus (C_2 \cup R')) \cup \{a\}$ is connected. 

    If we replace $A$ with the union of $A'$ and the segment $[x_2;\,x_3]$, and consider all new intersections as virtual crossings, we obtain a strong meander diagram of $K$.

\begin{figure}
\newcommand{\lenn}{0.5}
\newcommand{\mycol}{rgb:red,1;green,1;blue,4}
\newcommand{\mycoll}{rgb:red,4;green,1;blue,1}
    \centering
    \begin{tikzpicture}   
\useasboundingbox (0,-7*\lenn) rectangle (20*\lenn,4*\lenn);

\draw[line width = 1.5,color=\mycoll] (10.5*\lenn,0.4) to [out = -50, in = 90, looseness = 1] (11*\lenn,0);

\draw[line width = 3, white, double=black, double distance = 1.5] (1.49*\lenn, 0) to (2.51*\lenn, 0);
\draw[line width = 3, white, double=black, double distance = 1.5] (4.49*\lenn, 0) to (5.51*\lenn, 0);
\draw[line width = 3, white, double=black, double distance = 1.5] (6.49*\lenn, 0) to (7.51*\lenn, 0);
\draw[line width = 3, white, double=black, double distance = 1.5] (8.49*\lenn, 0) to (9.51*\lenn, 0);
    
\draw[line width = 3, white, double=\mycol, double distance = 1.5] (0,0) to [out = 90, in = 90, looseness = 1.2] (3*\lenn, 0)
    to  [out = -90, in = -90, looseness = 1.2] (13*\lenn, 0)
    to  [out = 90, in = 90, looseness = 1.2] (6*\lenn,0)
    to  [out = -90, in = -90, looseness = 1.2] (9*\lenn, 0)
    to  [out = 90, in = 90, looseness = 1.2] (10*\lenn,0)
    to  [out = -90, in = -90, looseness = 1.2] (5*\lenn, 0)
    to  [out = 90, in = 90, looseness = 1.2] (14*\lenn,0)
    to  [out = -90, in = -90, looseness = 1.2] (2*\lenn, 0)
    to  [out = 90, in = 90, looseness = 1.2] (1*\lenn,0)
    to  [out = -90, in = -90, looseness = 1.2] (15*\lenn, 0)
    to  [out = 90, in = 90, looseness = 1.2] (4*\lenn,0)
    to  [out = -90, in = -90, looseness = 1.2] (12*\lenn, 0)
    to  [out = 90, in = 90, looseness = 1.2] (7*\lenn,0)
    to  [out = -90, in = -90, looseness = 1.2] (8*\lenn, 0)
    to  [out = 90, in = 90+50, looseness = 1.2] (10.5*\lenn,0.4);

\draw[line width = 3, white, double=black, double distance = 1.5] (3.49*\lenn, 0) to (4.51*\lenn, 0);
\draw[line width = 3, white, double=black, double distance = 1.5] (5.49*\lenn, 0) to (6.51*\lenn, 0);
\draw[line width = 3, white, double=black, double distance = 1.5] (7.49*\lenn, 0) to (8.51*\lenn, 0);
    
\draw[line width = 1.5] (0,0) to (1.51*\lenn, 0);
\draw[line width = 1.5] (9.49*\lenn, 0) to (11*\lenn,0);
\draw[line width = 1.5] (2.49*\lenn, 0) to (3.51*\lenn, 0);

\draw[fill= white] (1*\lenn, 0) circle [radius = 0.07]; 
\draw[fill= white] (3*\lenn, 0) circle [radius = 0.07]; 
\draw[fill= white] (10*\lenn, 0) circle [radius = 0.07]; 

\draw[color = black, fill] (0, 0) circle [radius = 0.07]; 
\draw[color = black, fill] (11*\lenn, 0) circle [radius = 0.07];
\draw[color = black, fill] (10.5*\lenn, 0.4) circle [radius = 0.07];
\draw[color = black, fill] (18*\lenn, 0) circle [radius = 0.07];

\node (x1) at (-0.0, -0.3) {$x_1$};
\node (x2) at (11*\lenn 0, -0.3) {$x_2$};
\node (x3) at (18*\lenn 0, -0.3) {$x_3$};
\node (R) at (16.5*\lenn 0, 0.3) {$R$};
\node (A) at (11.27*\lenn 0, 0.3) {$A$};

\draw[line width = 1.2, dotted] (11*\lenn, 0) to (20*\lenn, 0);

\node (name) at (8*\lenn 0, -5.7*\lenn) {Strong semimeander diagram of $K$};

\end{tikzpicture}
\begin{tikzpicture}
\useasboundingbox (0,-6*\lenn) rectangle (20*\lenn,6*\lenn);
\draw[line width = 3, white, double=black, double distance = 1.5] (1.49*\lenn, 0) to (2.51*\lenn, 0);
\draw[line width = 3, white, double=black, double distance = 1.5] (4.49*\lenn, 0) to (5.51*\lenn, 0);
\draw[line width = 3, white, double=black, double distance = 1.5] (6.49*\lenn, 0) to (7.51*\lenn, 0);
\draw[line width = 3, white, double=black, double distance = 1.5] (8.49*\lenn, 0) to (9.51*\lenn, 0);
    
\draw[line width = 3, white, double=\mycol, double distance = 1.5] (0,0) to [out = 90, in = 90, looseness = 1.2] (3*\lenn, 0)
    to  [out = -90, in = -90, looseness = 1.2] (13*\lenn, 0)
    to  [out = 90, in = 90, looseness = 1.2] (6*\lenn,0)
    to  [out = -90, in = -90, looseness = 1.2] (9*\lenn, 0)
    to  [out = 90, in = 90, looseness = 1.2] (10*\lenn,0)
    to  [out = -90, in = -90, looseness = 1.2] (5*\lenn, 0)
    to  [out = 90, in = 90, looseness = 1.2] (14*\lenn,0)
    to  [out = -90, in = -90, looseness = 1.2] (2*\lenn, 0)
    to  [out = 90, in = 90, looseness = 1.2] (1*\lenn,0)
    to  [out = -90, in = -90, looseness = 1.2] (15*\lenn, 0)
    to  [out = 90, in = 90, looseness = 1.2] (4*\lenn,0)
    to  [out = -90, in = -90, looseness = 1.2] (12*\lenn, 0)
    to  [out = 90, in = 90, looseness = 1.2] (7*\lenn,0)
    to  [out = -90, in = -90, looseness = 1.2] (8*\lenn, 0)
    to  [out = 90, in = 130, looseness = 1.2] (10.5*\lenn,0.4);

\draw[line width = 3, white, double=black, double distance = 1.5] (3.49*\lenn, 0) to (4.51*\lenn, 0);
\draw[line width = 3, white, double=black, double distance = 1.5] (5.49*\lenn, 0) to (6.51*\lenn, 0);
\draw[line width = 3, white, double=black, double distance = 1.5] (7.49*\lenn, 0) to (8.51*\lenn, 0);
    
\draw[line width = 1.5] (0,0) to (1.51*\lenn, 0);
\draw[line width = 1.5] (2.49*\lenn, 0) to (3.51*\lenn, 0);
\draw[line width = 1.5] (9.49*\lenn, 0) to (11*\lenn,0);
\draw[line width = 1.5,color=\mycoll] (11*\lenn, 0) to (18*\lenn,0);

\draw[line width = 1.5,color=\mycoll] (10.5*\lenn,0.4) to [out = -50, in = 90, looseness = 2]  (8.5*\lenn, 0)
    to  [out = -90, in = -90, looseness = 1.2] (6.5*\lenn,0)
    to  [out = 90, in = 90, looseness = 1.2] (12.5*\lenn, 0)
    to  [out = -90, in = -90, looseness = 1.2] (3.5*\lenn,0)
    to  [out = 90, in = 90, looseness = 1.2] (18*\lenn, 0);

\draw[fill= white] (1*\lenn, 0) circle [radius = 0.07]; 
\draw[fill= white] (3*\lenn, 0) circle [radius = 0.07]; 
\draw[fill= white] (10*\lenn, 0) circle [radius = 0.07];
\draw[fill= white] (12*\lenn, 0) circle [radius = 0.07]; 
\draw[fill= white] (13*\lenn, 0) circle [radius = 0.07];
\draw[fill= white] (14*\lenn, 0) circle [radius = 0.07]; 
\draw[fill= white] (15*\lenn, 0) circle [radius = 0.07];

\draw[fill= white] (3.5*\lenn, 0) circle [radius = 0.07]; 
\draw[fill= white] (8.5*\lenn, 0) circle [radius = 0.07];
\draw[fill= white] (6.5*\lenn, 0) circle [radius = 0.07]; 
\draw[fill= white] (12.5*\lenn, 0) circle [radius = 0.07];

\draw[color = black, fill] (0, 0) circle [radius = 0.07]; 
\draw[color = black, fill] (11*\lenn, 0) circle [radius = 0.07];
\draw[color = black, fill] (10.5*\lenn, 0.4) circle [radius = 0.07];
\draw[color = black, fill] (18*\lenn, 0) circle [radius = 0.07];

\node (x1) at (-0.0, -0.3) {$x_1$};
\node (x3) at (18*\lenn 0, -0.3) {$x_3$};
\node (A) at (16*\lenn 0, 2.7*\lenn) {$A'$};

\node (name) at (8*\lenn 0, -5.7*\lenn) {Strong meander diagram of $K$};
\end{tikzpicture}
    \caption{An example of constructing a strong meander diagram from a strong semimeander one.}
    \label{fig: constructing meander diagram}
\end{figure}
\end{proof}
\begin{corollary}
    Each virtual knot has a meander diagram. 
\end{corollary}

\section{Virtual $k$-arc crossing number}
For a virtual knot $K$, let us denote by $\crv(K)$ the minimum number of classical crossings among all diagrams of $K$.

\begin{definition} \label{def : k-arc diagram}
Let $K$ be a virtual knot. A diagram $D$ of $K$ is called a \emph{$k$-arc diagram} if $D$ can be composed of $k$ smooth arcs such that all their self-intersections are virtual crossings.
\end{definition}

From Theorem~\ref{thm: strong semimeander diag exists}, it follows that for each virtual knot $K$ and for each integer $k>1$, there exists a $k$-arc diagram of $K$. Note that each 1-arc diagram is a diagram of the unknot.
%Является ли окружность дугой? 

\begin{definition}\label{def : k-arc crossing number}
Let $K$ be a virtual knot, and let $k$ be an integer. The \emph{virtual $k$-arc crossing number} (denoted by $\crv_k(K)$) is the minimal number of classical crossings among all $k$-arc diagrams of $K$ (if there are no $k$-arc diagrams of $K$ we assume $\crv_k(K) = \infty$).
\end{definition}

Definitions~\ref{def : k-arc diagram} and \ref{def : k-arc crossing number} can be reformulated in terms of Gauss diagrams. Indeed, a diagram $D$ of a virtual knot is a $k$-arc diagram if and only if the Gauss diagram constructed from $D$ can be composed of $k$ circle arcs in such a way that each chord starts on one arc and ends on the other.

For each virtual knot $K$, we can also consider the following invariant:
$$p^*(K) = \min \{k\,:\,\crv_k(K) = \crv(K)\}.$$ Thus we have the following sequence of inequalities:
\begin{equation*}
    \crv_2(K)\geq\crv_3(K)\geq \dots \geq \crv_{p^*(K)-1}(K) > \underbrace{\crv_{p^*(K)}(K)= \crv_{p^*(K)+1}(K) =\dots}_{=\crv(K)}\ . 
\end{equation*}

For a classical knot $K$, we can consider similar invariants, where the minimum is taken among diagrams with no virtual crossings. These invariants, denoted by  $\operatorname{cr}_k(K)$, were introduced in~\cite{B20}.
%Similar invariants for classical knots were introduced in~\cite{B20}.

\begin{thm}
For all classical knots $K$ and positive integer $k$ we have $\crv_k(K) = \operatorname{cr}_k(K)$.
\end{thm}

\begin{proof} 
Manturov projection Theorem~\cite{M1} says that there is a natural projection from the class of virtual diagrams of a knot $K$ to the class of classical diagrams of the same knot type. It consists of the sequence of operations erasing the chords in the Gauss diagram or which is the same changing some of the classical crossings to the virtual ones.

Let $K$ be a classical knot, let $D$ be its $k$-arc diagram such that the number of classical crossings in $D$ is equal to the $\crv_k(K)$, and let $C$ be the Gauss diagram corresponding to $D$. 
According to the Manturov projection theorem, we can obtain a new Gauss diagram $C'$ corresponding to a diagram of $K$ without virtual crossings by erasing chords from $C$. Since each chord in a Gauss diagram corresponds to a classical crossing, $C'$ has no more crossings than $C$. 
Moreover, if $C$ was composed of $k$ circular arcs, this property is preserved after removing a chord, so $C'$ is also composed of $k$ circular arcs (though some arcs may no longer have chords starting or ending on them). 
Therefore, we have $\crv_k(K)\geq \operatorname{cr}_k(K)$. The other inequality is obvious, so the two quantities are equal. 
%
%If we erase a chord that is not a beginning or ending point of one of the $k$ arcs the crossing number becomes smaller. If this was a starting or ending point of one of the $k$-arcs then this arc disappears and merges with another arc but we can introduce an extra arc with no chords ending or starting within it.
%
%Thus at the end of the sequence of operations defining Manturov projection, we get the diagram of the same knot with $k$ arcs and a crossing number that is less or equal to the one for the virtual knot diagram. Thus for all $k$ we have $\crv_k(K)\geq \operatorname{cr}_k(K)$. The other inequality is obvious so the two quantities are equal.
\end{proof}

\begin{thm}
    For each virtual knot $K$ and integer $k\geq2$ the following inequality holds
    $$
    \crv_{k}(K) \leq \crv_{k+1}(K) + 2\frac{\left(\crv_{k+1}(K) \right)^2}{(k+1)^2}. 
    $$
\end{thm}
\begin{proof}
    Let $D$ be a $(k+1)$-arc diagram of $K$ such that the number of classical crossings in $D$ is equal to the $\crv_{k+1}(K)$. Let us choose a partition of $D$ into $k+1$ arcs with no classical self-crossings, and let $J$ be one of these $k+1$ arcs with the smallest number $n$ of classical crossings on it (i.\,e. $n \leq 2\frac{\crv_{k+1}(K)}{k+1}$). Since $k + 1 \geq 3$, the arcs $J_l$ and $J_r$ neighboring $J$ are distinct. Without loss of generality, we assume that the number of classical crossings $m$ between $J$ and $J_r$ does not exceed the number of classical crossings between $J$ and $J_l$ (i.\,e. $m \leq n/2$). 

    \begin{figure}[h]
    \newcommand{\loopangle}{50}
    \newcommand{\virtcrosradius}{0.06}
        \centering
        \begin{tikzpicture}
            \node (JR) at (4.3, 0.3) {$J_r$};
            \node (JL) at (0.3, 0.3) {$J_l$};
            \node (x) at (3.2, 0.25) {$x$};
            
            \node (j_l1) at (0, 0) {};
            \node (j_l2) at (1, 0) {};
            
            \node (j_top) at (2, 2.5) {};
            \node (j_top_l) at (1.3, 1.1) {};
            \node (j_top_r) at (2.7, 1.1) {};
            \node (j_2) at (3, 0) {};
            
            \node (j_r1) at (4, 0){};
            \node (j_r2) at (5, 0){};

            \node (blue1) at (0.5, 2.2) {};
            \node (blue2) at (3.5, 2.2){};
            
            \node (b1) at (3, 1) {};
            \node (b2) at (3, -0.5) {};

            \node (x1) at (3.5, 1) {};
            \node (x11) at (3.5, 0.5) {};
            \node (x2) at (3.5, -0.5) {};

        \draw[line width = 1.5, blue] (j_l1.center) to (j_l2.center);
        
        \draw[line width = 1.5] (j_l2.center) 
            to [out = 0, in = 180+\loopangle, looseness = 1] (j_top_r.center);
            
        \draw[line width = 1.5, black]  (j_top_r.center)
            to [out = \loopangle, in = 0, looseness = 1.5] (j_top.center);
        \draw [line width = 3, white, double=blue, double distance = 1.5] (blue1.center) 
           to [out = -45, in =225, looseness = 1] (blue2.center);
        \draw[line width = 3, white, double=black, double distance = 1.5]  (j_top.center)
            to [out = 180, in = 180-\loopangle, looseness = 1.5] (j_top_l.center);
        \draw[line width = 1.5, black]  (j_top_l.center)
            to [out = -\loopangle, in = 180, looseness = 1] (j_2.center)
            to (j_r1.center);
        
        \draw [line width = 1.5, red] (j_r1.center) to (j_r2.center);
        
        \draw [line width = 3, white, double=red, double distance = 1.5] (x1.center) to (x2.center);  
        \draw[fill= white] (2, 0.353) circle [radius = \virtcrosradius]; 
        \end{tikzpicture}
        \raisebox{0.8cm}{
\scalebox{0.88}{
\begin{tikzpicture}
\draw [white, ultra thick] (,0) circle [radius=0.1];;
\node(r) at (-0.15, -1.7) {}; 
\draw[->, ultra thick] (-1.5,-0.5)--(1.2,-0.5);
\end{tikzpicture}}}
        \begin{tikzpicture}
            \node (x) at (0.3, -0.25) {$x$};
            
            \node (j_l1) at (0, 0) {};
            \node (j_l2) at (1, 0) {};
            
            \node (j_top) at (2, 2.5) {};
            \node (j_top_l) at (1.3, 1.1) {};
            \node (j_top_r) at (2.7, 1.1) {};
            \node (j_2) at (3, 0) {};

            \node (r_top) at (2, 2.3) {};
            \node (r_top2) at (2, 2.8) {};
            \node (r_top_l) at (1.6, 1.1) {};
            \node (r_top_l2) at (0.9, 1.3) {};
            \node (r_top_r) at (2.4, 1.1) {};
            \node (r_top_r2) at (3.1, 1.3) {};
            
            \node (r_2) at (0.8, 0.25) {};
            \node (r_down) at (0.8, -0.25) {};
            
            \node (j_r1) at (4, 0){};
            \node (j_r2) at (5, 0){};

            \node (blue1) at (0.5, 2.2) {};
            \node (blue2) at (3.5, 2.2){};
            
            \node (b1) at (3, 1) {};
            \node (b2) at (3, -0.5) {};

            \node (x1) at (3.5, 1) {};
            \node (x11) at (3.2, 0.25) {};
            \node (x21) at (3.2, -0.25) {};
            \node (x2) at (3.5, -0.5) {};

        \draw[line width = 1.5, blue] (j_l1.center) to (j_l2.center);
        
        \draw[line width = 1.5] (j_l2.center) 
            to [out = 0, in = 180+\loopangle, looseness = 1] (j_top_r.center);           
        \draw[line width = 1.5, black]  (j_top_r.center)
            to [out = \loopangle, in = 0, looseness = 1.5] (j_top.center);
        \draw [line width = 3, white, double=blue, double distance = 1.5] (blue1.center) 
           to [out = -45, in =225, looseness = 1] (blue2.center);
        \draw[line width = 3, white, double=black, double distance = 1.5]  (j_top.center)
            to [out = 180, in = 180-\loopangle, looseness = 1.5] (j_top_l.center);
        \draw[line width = 1.5, black]  (j_top_l.center)
            to [out = -\loopangle, in = 180, looseness = 1] (j_2.center)
            to (j_r1.center);

         \draw[line width = 1.5, red]  (x1.center)
            to [out = -90, in = 0, looseness = 1] (x11.center);
        \draw[line width = 1.5, red]  (x11.center)
            to [out = 180, in = -\loopangle, looseness = 1] (r_top_l.center);
        \draw[line width = 3, white, double=red, double distance = 1.5]  (r_top_l.center)
            to [out = 180-\loopangle, in = 180, looseness = 1] (r_top.center);
        \draw[line width = 3, white, double=red, double distance = 1.5]  (r_top.center)
            to [out = 0, in =\loopangle, looseness = 1] (r_top_r.center);
        \draw[line width = 1.5, red]  (r_top_r.center)
            to [out = 180+\loopangle, in = 0, looseness = 1] (r_2.center);
        \draw[line width = 3, white, double=red, double distance = 1.5]  (r_2.center)
            to [out = 180, in = 180, looseness = 1] (r_down.center);
        
        \draw[line width = 1.5, red]  (r_down.center)
            to [out = 0, in = 180+\loopangle, looseness = 1] (r_top_r2.center);
        \draw[line width = 3, white, double=red, double distance = 1.5]  (r_top_r2.center)
            to [out = \loopangle, in = 0, looseness = 1] (r_top2.center);
        \draw[line width = 3, white, double=red, double distance = 1.5]  (r_top2.center)
            to [out = 180, in = 180-\loopangle, looseness = 1] (r_top_l2.center);
        \draw[line width = 1.5, red]  (r_top_l2.center)
            to [out = -\loopangle, in = 180, looseness = 1] (x21.center);
         \draw[line width = 1.5, red]  (x21.center)
            to [out = 0, in = 90, looseness = 1] (x2.center);
        
        \draw [line width = 1.5, red] (j_r1.center) to (j_r2.center);
        \draw[fill= white] (2, 0.353) circle [radius = \virtcrosradius]; 
        \draw[fill= white] (2, 0.7) circle [radius = \virtcrosradius]; 
        \draw[fill= white] (2, 0.15) circle [radius = \virtcrosradius]; 

        \draw[fill= white] (1.85, 0.245) circle [radius = \virtcrosradius]; 
        \draw[fill= white] (2.15, 0.245) circle [radius = \virtcrosradius]; 

        \draw[fill= white] (1.79, 0.545) circle [radius = \virtcrosradius]; 
        \draw[fill= white] (2.21, 0.545) circle [radius = \virtcrosradius]; 

        \draw[fill= white] (1.64, 0.446) circle [radius = \virtcrosradius]; 
        \draw[fill= white] (2.36, 0.446) circle [radius = \virtcrosradius]; 
        \end{tikzpicture}
        
        \caption{Pulling crossing $x$ to $J_l$}
        \label{fig: k-arc inequality}
    \end{figure}

   Now, move along $J$ starting from its endpoint adjacent to $J_l$, until reaching the first classical crossing $x$, where $J$ and $J_r$ intersect. We can use a sequence of $\mathfrak{X}$-moves to pull $x$ to $J_l$. All new crossings will lie on $J_r$. Importantly, this operation does not create any new classical crossings between $J$ and $J_r$, or new self-intersections of $J_r$ that are classical crossings. This is because all intersections between these arcs before $x$ were virtual crossings, and $\mathfrak{X}$-move through a virtual crossing adds only virtual crossings (see Fig.~\ref{fig: x-moves}). Note that the number of virtual crossings that are intersections between $J$ and $J_r$ as well as between $J_r$ and $J_r$ may increase --- see the example in Fig.~\ref{fig: k-arc inequality}. After repeating this procedure for all classical crossings between $J$ and $J_r$ (overall $m$ times), we obtain a new diagram $D'$ of $K$ in which $J$ and $J_r$ have no common classical crossings. Consequently, $D'$ is actually a $k$-arc diagram. Moreover, each $\mathfrak{X}$-move through a classical crossing adds exactly 2 classical crossings, and so we added no more than $2(n-m)$ classical crossings during the whole process. Thus, we obtain the following inequality:
    $$
    \crv_{k}(K) \leq  \crv_{k+1}(K) + 2m\left(n-m \right) \leq \crv_{k+1}(K) + \frac{n^2}{2} \leq \crv_{k+1}(K) + 2\frac{\left(\crv_{k+1}(K) \right)^2}{(k+1)^2}.
    $$    
\end{proof}
\begin{remark}
    The same inequality holds for classical $k$-arc crossing numbers, as it is proved in~\cite{B20}.
\end{remark}

\bibliography{biblio}
\bibliographystyle{alpha}
\end{document}